\documentclass[letterpaper]{amsart}
\usepackage[showframe=false, margin = 1.45 in]{geometry}
\usepackage[stretch=10,shrink=10]{microtype}
\usepackage{amssymb}
\usepackage{eucal}
\usepackage{tikz}
\usepackage{ifthen}
\usetikzlibrary{decorations.pathmorphing}
\usepackage{enumerate}
\usepackage{constants}
\usepackage{hyperref}

\newcommand{\abs}[1]{\left\lvert #1 \right\rvert}

  \newcommand{\eps}{\epsilon}
  \newcommand{\GGG}{\mathfrak{G}}
   \newcommand{\Bb}{\mathcal{B}}

  \newcommand{\Ee}{\mathcal{E}}
  \newcommand{\Ff}{\mathcal{F}}
  \newcommand{\Ii}{\mathcal{I}}
  \newcommand{\Jj}{\mathcal{J}}
  \newcommand{\Kk}{\mathcal{K}}

  \newcommand{\RR}{\mathbb{R}}

  \newcommand{\NN}{\mathbb{N}}

  \newcommand{\limn}{\lim_{n\to\infty}}

  \newcommand{\floor}[1]{\left\lfloor #1 \right\rfloor}

  \renewcommand{\subset}{\subseteq}

  \newcommand{\E}{\mathbf{E}}
  \newcommand{\I}{\mathbf{I}}
  \newcommand{\Z}{\mathbf{Z}}
  \newcommand{\Poi}{\mathrm{Poi}}

  \newcommand{\dtv}{d_{TV}}
  \renewcommand{\P}{\mathbf{P}}
  
  \newcommand{\1}{\mathbf{1}}

\renewconstantfamily{normal}{symbol=c}

\begin{document}
\newcommand{\Cy}[1]{C_{#1}}
\newcommand{\Cyinf}[1]{C^{(\infty)}_{#1}}
\newcommand{\badw}[2][n]{B^{(#1)}_{#2}}

\theoremstyle{plain}
\newtheorem{thm}{Theorem}
\newtheorem{lemma}[thm]{Lemma}
\newtheorem{prop}[thm]{Proposition}
\newtheorem{cor}[thm]{Corollary}
\newtheorem{clm}[thm]{Claim}

\theoremstyle{definition}
\newtheorem{defn}[thm]{Definition}

\theoremstyle{remark}
\newtheorem{rmk}[thm]{Remark}
\newtheorem{exm}[thm]{Example}
\tikzset{vert/.style={circle,fill,inner sep=0,
    minimum size=0.12cm,draw}}

\title[Exchangeable pairs, switchings, and random regular graphs]{Exchangeable pairs, switchings,\protect\\and random regular graphs}
\author{Tobias Johnson}
\address{University of Southern California\\Department of Mathematics\\3620 S.~Vermont Ave, KAP~108\\
  Los Angeles, CA 90089}
\email{tobias.johnson@usc.edu}
\keywords{Switchings, Stein's method, exchangeable pairs, random regular graphs,
linear eigenvalue statistics}
\subjclass[2010]{05C80, 60B10, 60B20}

\thanks{The author acknowledges support from the NSF by grants
DMS-0847661 and DMS-1401479}
\begin{abstract}
  We consider the distribution of cycle counts in a random regular
  graph, which is closely linked to the graph's spectral properties.
  We broaden the asymptotic regime in which the cycle counts
  are known to be
  approximately Poisson, and we give an explicit bound
  in total variation distance for the approximation.  Using this result,
  we calculate limiting distributions of linear eigenvalue statistics
  for random regular graphs.
  
  Previous results on the distribution of
  cycle counts by McKay, Wormald, and Wysocka (2004) used the method of
  switchings, a combinatorial technique for asymptotic enumeration.  Our
  proof uses Stein's method of exchangeable pairs and demonstrates
  an interesting connection between the two techniques.
\end{abstract}

\maketitle

\section{Introduction}
Suppose that $\lambda_1,\ldots,\lambda_n$
are the eigenvalues of an $n\times n$ random matrix.
The random variable $\sum_{i=1}^n f(\lambda_i)$ for a given
function $f$ is known as a \emph{linear eigenvalue statistic},
and it is a common object of study in random matrix theory, typically
as $n$ tends to infinity.

Let $G$ be chosen uniformly at random from the space
of all simple $d$-regular graphs on $n$ vertices, and consider
its adjacency matrix.
Brendan McKay determined the first-order behavior of its linear eigenvalue
statistics, showing that $n^{-1}\sum_{i=1}^n f(\lambda_i)$ converged
in probability to a deterministic limit as $n\to\infty$ \cite{McK}.  
In \cite{DJPP},
the second-order behavior of linear eigenvalue statistics was computed
for a slightly different model of random regular graph, with improved results
given in \cite[Chapter~3]{Elliot}.
The motivating goal of this paper is to prove similar results
for uniformly chosen random regular graphs, 
which we carry out in Theorems~\ref{thm:dfixedlimit} and~\ref{thm:dgrowslimit}.

We will defer further discussion of this problem and its background
until Section~\ref{sec:eigenvalues}.  Until then, we discuss
several combinatorial and probabilistic results interesting
in their own right that we will
achieve along the way.
Let $\Cy{k}$ denote the number of cycles of length
$k$ in the random regular graph $G$.
The distribution of these random variables has been studied since 
\cite{Bol80,Wor81}, where it was proven that
 $(\Cy{3},\ldots,\Cy{r})$ converges in law
to a vector of independent Poisson random variables as $n$ tends to infinity,
with $r$ held fixed.  
As early as \cite{McK}, the cycle counts
of a graph have been used to investigate properties of the graph's
eigenvalues.  We take this approach as well, converting our original
problem into one of accurately estimating the distribution of this random
vector.

The strongest results on the cycle counts of a random regular graph came
in \cite{MWW}, where the Poisson approximation was shown to hold even
as $d=d(n)$ and $r=r(n)$ grow with $n$, so long as $(d-1)^{2r-1}=o(n)$.
This is a natural boundary: in this asymptotic regime, 
all cycles in $G$ of length $r$ or less have disjoint edges,
asymptotically almost surely. If $(d-1)^{2r-1}$ grows any
faster, this fails.  This led the authors in \cite{MWW} to speculate that the
Poisson approximation failed beyond this threshold.
Surprisingly, this is not the case. 
In Theorem~\ref{thm:bestpoiapprox}, we give
a Poisson approximation for the cycle counts
that holds so long as $\sqrt{r}(d-1)^{\frac32r-1}=o(n)$.
We give a quantitative bound on the accuracy
of the approximation, which is
the necessary ingredient for our results on linear eigenvalue statistics.
As a bonus, we give in Theorem~\ref{thm:genpoiapprox}
a distributional approximation
not just of the cycle counts, but of a more general process defined by the cycles.

The Poisson approximation in \cite[Theorem~1]{MWW} uses
a combinatorial technique for asymptotic enumeration known as
the method of switchings.  We adapt this technique to use Stein's method
of exchangeable pairs for Poisson approximation.  We discuss both methods
further in the following section.  As noted in \cite{Wor}, they
have some obvious similarity, but 
we believe that this is the first time they
have been connected in a rigorous way.  This connection gives a novel
construction of an exchangeable pair for use with Stein's method, and
it allows the machinery of Stein's method to be used in
some new combinatorial settings.

In Section~\ref{sec:prelims}, we give some basic definitions
and preliminary estimates
on random regular graphs.
Section~\ref{sec:poisson}
presents our Poisson approximation.
The core argument and the most general result
is Theorem~\ref{thm:genpoiapprox}, and our main result on cycle counts is
Theorem~\ref{thm:bestpoiapprox}.
In Section~\ref{sec:eigenvalues}, we give the context and proofs of our results
on linear
eigenvalue statistics of random regular graphs.
\subsection{Switchings and Stein's method}
The method of switchings,  pioneered
by Brendan McKay and Nicholas Wormald,
has been
applied to asymptotically enumerate combinatorial structures
that defy exact counts, including
Latin rectangles \cite{GM} and matrices with prescribed row and column sums
\cite{Mc84,MW,GMW}.  It has seen its biggest use in analyzing regular graphs;
see \cite{KSVW}, \cite{MWW}, \cite{KSV}, and \cite{BSK} for some examples.
A good summary of switchings in random regular graphs
can be found in Section~2.4 of \cite{Wor99}.
    
    The basic idea of the method is to choose
    two families of objects, $A$ and $B$, and investigate only their
    relative sizes.  To do this, one defines a set
    of switchings that
    connect elements of $A$ to elements of $B$.
    If every element of $A$ is connected to
    roughly $p$ objects in $B$,
    and every element in $B$ is connected to roughly
    $q$ objects in $A$, then 
    by a double-counting argument, $|A|/|B|$
    is approximately $q/p$.  When the objects in question
    are elements of a probability space, this gives an estimate of
    the relative probabilities of two events.

Stein's method (sometimes
called the Stein-Chen method when used for Poisson approximation)
is a powerful and elegant tool to compare two probability
distributions.  It was originally developed by Charles
Stein for normal approximation; its first published
use is \cite{St1}. Louis Chen adapted the method 
for Poisson approximation \cite{Chen1}.  Since then,
Stein, Chen, and a score of others have adapted Stein's method to
a wide variety of circumstances.  The survey
paper \cite{Ross} gives a broad introduction to Stein's method, 
and \cite{BHJ} and \cite{CDM} 
focus specifically on using it for Poisson approximation.

We will use the technique of exchangeable pairs,
following the treatment in \cite{CDM}.
Suppose we want to bound
the distance of the law of $X$ from the Poisson distribution.
The technique is to introduce an auxiliary randomization to $X$ to
get a new random variable $X'$ so that $X$ and $X'$ are exchangeable
(that is, $(X,X')$ and $(X', X)$ have the same law).
If $X$ and $X'$ have the right relationship---specifically,
if they behave
like two steps in an immigration-death 
process whose stationary distribution
is Poisson---then Stein's method
gives an easy proof that $X$ is approximately Poisson.

    Switchings and Stein's method have bumped into each other
    several times.
For instance,
both techniques have been used to study Latin rectangles \cite{Stein,GM},
and the analysis of random contingency tables in \cite{DS} is similar
to combinatorial work like \cite{GrM}.
Nevertheless, we believe that this is the first explicit connection
between the two techniques.
The essential idea is to use a random switching as the auxiliary randomization
in constructing an exchangeable pair.  

We believe the connection between switchings and Stein's method
may prove profitable to users of both techniques.  
Using Stein's method in conjunction with a switching argument
allows for a quantitative bound  on the accuracy of the approximation.
Stein's method
can also be used for approximation
by other distributions besides Poisson and for proving
concentration bounds  (see \cite{Cha}).
On the other hand, Stein's method cannot prove results as sharp
as \cite[Theorem~2]{MWW}, which gives an extremely accurate bound
on the probability that a random graph
has no cycles of length $r$ or less.
The bare-hands switching arguments used there might
be useful to anyone who needs a particularly sharp bound on
a Poisson approximation at a single point.

\section{Preliminaries}\label{sec:prelims}
A \emph{$d$-regular graph} is one for which all vertices have
degree exactly $d$.  We call a graph \emph{simple} if it has no
loops (edges between a vertex and itself) or parallel edges.
By \emph{random $d$-regular graph on $n$ vertices}, we mean
a random graph chosen uniformly from the space of all simple $d$-regular
graphs on $n$ vertices (unless we specifically refer to another model).
When $d$ is odd, we always assume that $n$ is even,
since there are no $d$-regular graphs on $n$ vertices with $d$ and $n$ odd.
By \emph{cycle}, we mean what is sometimes called a simple cycle:
a walk on a graph starting and ending at the same vertex, and with no
repeated edges or vertices along the way.  
For vertices $u$ and
$v$ in a graph, we will use the notation
$u\sim v$ to denote that the edge $uv$ exists.  The distance between
two vertices is the length of the shortest path between them,
and the distance between two sets of vertices is the shortest
distance between a vertex in one set and a vertex in the other.

Here and throughout, we will use $c_1,\,c_2,\ldots$
to denote absolute constants whose values are unimportant to us.
\begin{prop}\label{prop:McKayestimate}
  Let $G$ be a random $d$-regular graph on $n$ vertices,
  with $d\leq n^{1/3}$.
  \begin{enumerate}[(a)]
    \item \label{item:generalsubgraph}
      Suppose $H$ is a subgraph of the complete graph
      $K_n$ in which every vertex has degree $2$ or higher.
      Let $e$ be the number of edges and $v$ the number of vertices in $H$.
      Suppose $e\leq 2n^{1/10}$. Then
      \begin{align*}
        \P[H\subseteq G]\leq \frac{\Cl{generalsubgraph}(d-1)^e}{n^e}.
      \end{align*}
    \item\label{item:onecycle}
      Let $\alpha$ be a cycle of length $k\leq 2n^{1/10}$ in the complete graph
      $K_n$.  Then
      \begin{align*}
        \P[\alpha\subset G]&\leq \frac{\Cr{generalsubgraph}(d-1)^k}{n^k}.
      \end{align*}
    \item \label{item:twocycles}
      Let $\beta$ be another cycle in $K_n$ of length $j\leq n^{1/10}$,
      and suppose that $\alpha$ and $\beta$ share $f$ edges.
      Then
      \begin{align*}
        \P[\alpha\cup\beta\subset G]&\leq \frac{\Cr{generalsubgraph}(d-1)^{j+k-f}}{n^{j+k-f}}.
      \end{align*}
  \end{enumerate}
\end{prop}

\begin{proof}
  Statements~(\ref{item:onecycle}) and~(\ref{item:twocycles}) are specializations of
  (\ref{item:generalsubgraph}), which follows directly from Theorem~3a in \cite{MWW}.
\end{proof}

\section{Poisson approximation of cycle counts by Stein's method}
\label{sec:poisson}
\subsection{Stein's method background}
\label{subsec:stein1}

The main idea of Stein's method of exchangeable pairs
is to perturb a random variable $X$ to get a new random variable $X'$,
and then to examine
the relationship between the two.
The basic heuristic is that
if $(X,X')$ is exchangeable and
\begin{align*}
  \P[X'=X+1\mid X]&\approx \frac{\lambda}{c},\\
  \P[X'=X-1\mid X]&\approx \frac{X}{c},
\end{align*}
for some constant $c$, then 
$X$ is approximately Poisson with mean $\lambda$.
(When $X$ and $X'$ are 
two steps in a stationary immigration-death chain whose 
invariant distribution is Poisson with mean $\lambda$, these
equations hold exactly.)  The following proposition
gives a precise, multivariate version of this heuristic.
Recall that the total variation distance between the laws of two
random variables $X$ and $Y$ taking values in $\NN=\{0,1,2,\ldots\}$ is
given by
\begin{align*}
  \dtv(X,Y):= \sup_{A\subset\NN}\abs{\P[X\in A] - \P[Y\in A]}.
\end{align*}
\begin{prop}[{\cite[Proposition~10]{CDM}}]\label{prop:steinsmethod}
  Let $W=(W_1,\ldots,W_r)$ be a random vector taking values
  in $\NN^r$, and let the coordinates of
  $Z=(Z_1,\ldots,Z_r)$ be independent Poisson random variables with
  $\E Z_k=\lambda_k$.  Let $W'=(W_1',\ldots,W_r')$ be defined
  on the same space as $W$, with $(W,W')$ an exchangeable pair.
  
  For any choice of $\sigma$-algebra $\Ff$ with
  respect to which $W$ is measurable and any
  choice of constants $c_k$,
  \begin{align*}
    \dtv(W, Z)\leq
      \sum_{k=1}^r\xi_k\Big(\E\big|\lambda_k-c_k\P[\Delta^+_k\mid\Ff]\big|
      +\E\big| W_k-c_k\P[\Delta^-_k\mid\Ff]  \big|
      \Big),
  \end{align*}
  with $\xi_k=\min(1,1.4\lambda_k^{-1/2})$ and
  \begin{align*}
    \Delta^+_k &= \{W_k'=W_k+1,\ \text{$W_j=W'_j$ for $k<j\leq r$}\},\\
    \Delta^-_k &= \{W_k'=W_k-1,\ \text{$W_j=W'_j$ for $k<j\leq r$}\}.
  \end{align*}
\end{prop}
\begin{rmk}
  We have changed the statement of the proposition from
\cite{CDM} in two small ways:
we condition our probabilities on $\Ff$, rather than on $W$, and
we do not require that $\E W_k =\lambda_k$ (though
the approximation will fail if this is far from true).  Neither
change invalidates the proof of the proposition.
\end{rmk}

\begin{rmk}
There is a direct connection between switchings and a certain bare-hands
version of Stein's method.  Though this is not
what we use in this paper,
it is helpful in understanding why Stein's method
and the method of switchings are so similar.
If $(X,X')$ is exchangeable, then as explained in
\cite[Section~2]{St92},
one can  directly investigate
ratios of probabilities of different values of $X$ using the equation
\begin{align*}
  \frac{\P[X=x_1]}{\P[X=x_2]} &= \frac{\P[X'=x_1\mid X=x_2]}{\P[X'=x_2\mid
  X=x_1]}.
\end{align*}
This technique
bears a strong resemblance to the method of switchings: if we think
of $X$ as some property of a random graph (for example, number of cycles)
and $X'$ as that property after a random switching has been applied, then
this formula instructs us to count how many switchings change $X$
from $x_1$ to $x_2$ and vice versa, just as one does
when using switchings for asymptotic enumeration. 
\end{rmk}

\newcommand{\Gcond}{\widetilde{G}}
\newcommand{\Cycond}[1]{\widetilde{C}_{#1}}

\subsection{Counting switchings}
We start by defining our switchings.
Besides some small notational differences,
the definitions will be the same as those in \cite{MWW}.
To avoid repetition
 of the phrase ``cycles of length $r$ or less,'' we
 will refer to such cycles as \emph{short}.

\begin{figure}
  \begin{center}
    \begin{tikzpicture}[scale=1.2]
      \begin{scope}[bend angle=20]
      \fill[black!10, rounded corners] (-0.6,-0.5) rectangle (3.6, 1.7);
      \draw
            (0,1.2) node[vert,label=above:$v_0$] (v0) {}
            ++(1,0) node[vert,label=above:$v_1$] (v1) {}
            ++(1,0) node[vert,label=above:$v_2$] (v2) {}
            ++(1,0) node[vert,label=above:$v_3$] (v3) {}
            (-0.15,0) node[vert,label={[xshift=-0.09cm]below:$u_0$}] (u0) {}
            (0.15,0) node[vert,label={[xshift=0.09cm]below:$w_0$}] (w0) {}
            (0.85,0) node[vert,label={[xshift=-0.09cm]below:$u_1$}] (u1) {}
            (1.15,0) node[vert,label={[xshift=0.09cm]below:$w_1$}] (w1) {}
            (1.85,0) node[vert,label={[xshift=-0.09cm]below:$u_2$}] (u2) {}
            (2.15,0) node[vert,label={[xshift=0.09cm]below:$w_2$}] (w2) {}
            (2.85,0) node[vert,label={[xshift=-0.09cm]below:$u_3$}] (u3) {}
            (3.15,0) node[vert,label={[xshift=0.09cm]below:$w_3$}] (w3) {};
        \draw[thick] (v0)--(v1)--(v2)--(v3);
        \draw[thick, bend right] (v0) to (v3);   
        \draw[thick] (w0)--(u1) (w1)--(u2) (w2)--(u3) (u0) to[bend left] (w3);  
      \end{scope}
      \draw[<->, decorate, decoration={snake,amplitude=.4mm,
            segment length=2mm,post length=1.5mm,pre length=1.5mm}] 
            (3.7, 0.6)--(5.5,0.6);
      \begin{scope}[xshift=6.2cm]
        \fill[black!10, rounded corners] (-0.6,-0.5) rectangle (3.6, 1.7);
      \draw (0,1.2) node[vert,label=above:$v_0$] (v0p) {}
            ++(1,0) node[vert,label=above:$v_1$] (v1p) {}
            ++(1,0) node[vert,label=above:$v_2$] (v2p) {}
            ++(1,0) node[vert,label=above:$v_3$] (v3p) {}
            (-0.15,0) node[vert,label={[xshift=-0.09cm]below:$u_0$}] (u0p) {}
            (0.15,0) node[vert,label={[xshift=0.09cm]below:$w_0$}] (w0p) {}
            (0.85,0) node[vert,label={[xshift=-0.09cm]below:$u_1$}] (u1p) {}
            (1.15,0) node[vert,label={[xshift=0.09cm]below:$w_1$}] (w1p) {}
            (1.85,0) node[vert,label={[xshift=-0.09cm]below:$u_2$}] (u2p) {}
            (2.15,0) node[vert,label={[xshift=0.09cm]below:$w_2$}] (w2p) {}
            (2.85,0) node[vert,label={[xshift=-0.09cm]below:$u_3$}] (u3p) {}
            (3.15,0) node[vert,label={[xshift=0.09cm]below:$w_3$}] (w3p) {};
        \draw[thick] (u0p)--(v0p)--(w0p)
                     (u1p)--(v1p)--(w1p)
                     (u2p)--(v2p)--(w2p)
                     (u3p)--(v3p)--(w3p);
      \end{scope}
    \end{tikzpicture}
  \end{center}
  \caption{The change from left to right
  is a \emph{forward switching}, and from right to left is
  a \emph{backward switching}.}
  \label{fig:switchings}
\end{figure}
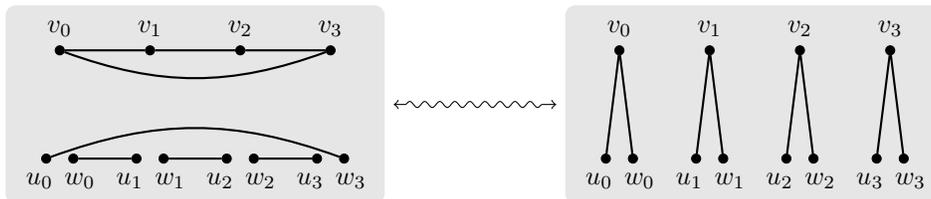

Let $G$ be a $d$-regular graph.
Suppose that $\alpha=v_0\cdots v_{k-1}$ is a cycle in $G$, and
let $e_i=v_iv_{i+1}$, interpreting
all indices modulo $k$ from now on.  Let $e_i'=w_iu_{i+1}$ for $0\leq i\leq k-1$
be oriented edges such that neither $u_i$ nor $w_i$ is adjacent to
$v_i$.  Consider the act of deleting these $2k$ edges and replacing them
with
the edges $v_iu_i$ and $v_iw_i$ for $0\leq i\leq k-1$
to obtain a new $d$-regular graph $G'$ with the cycle $\alpha$
deleted
(see Figure~\ref{fig:switchings}).  We call this action
induced given by the sequences $(v_i)$, $(u_i)$, and $(w_i)$
a \emph{forward $\alpha$-switching}.   We will consider forward 
$\alpha$-switchings
only up to cyclic rotation of indices; that is, we identify the $2k$
different $\alpha$-switchings obtained by cyclically rotating
all sequences $v_i$, $u_i$, and $w_i$.

To go the opposite direction, suppose $G$ contains oriented paths 
$u_iv_iw_i$ for $0\leq i\leq k-1$
such that $v_i\not\sim v_{i+1}$ and $w_i\not\sim u_{i+1}$.
Consider the act of deleting all edges $u_iv_i$ and $v_iw_i$
and replacing them with $v_iv_{i+1}$ and $w_iu_{i+1}$ for all
$0\leq i\leq k-1$ to create a new graph $G'$ that contains
the cycle $\alpha=v_0\cdots v_{k-1}$.  We call this a \emph{backwards
$\alpha$-switching}.
  Again, we consider switchings
only up to cyclic rotation of all indices.

We call an $\alpha$-switching \emph{valid} if  $\alpha$
 is the
only short cycle created or destroyed by the switching.
For each valid forward $\alpha$-switching taking $G$ to
$G'$, there is a corresponding valid backwards $\alpha$-switching
taking $G'$  to $G$.
Let $F_\alpha$ and $B_\alpha$
be the number of valid forward and backwards $\alpha$-switchings,
respectively, on some graph $G$.  
Using arguments drawn from \cite[Lemma~3]{MWW}, we give
some estimates on them.
\begin{lemma}\label{lem:forwardswitchings}
  Let $G$ be a deterministic $d$-regular graph on $n$ vertices
  with cycle counts $\{\Cy{k},\,k\geq 3\}$.
  For any short cycle $\alpha\subset G$ of length $k$,  
  \begin{align}
    F_\alpha\leq [n]_kd^k.\label{eq:fub}
  \end{align} 
  If $\alpha$ does not
  share an edge with another short cycle,
  \begin{align}
    F_\alpha&\geq
      [n]_kd^k\left(1- \frac{2k\sum_{j=3}^rj\Cy{j}+\Cl{4}k(d-1)^r}{nd}\right).
      \label{eq:flb}
   \end{align}
\end{lemma}
\begin{proof}
  The question is, with $\alpha=v_0\cdots v_{k-1}$
  and $e_i=v_iv_{i+1}$ given, how many ways are there to choose
  $e_0',\ldots,e_{k-1}'$ that give
  a valid switching?
  There are at most $[n]_kd^k$ choices of oriented
  edges $e_0',\ldots,e_{k-1}'$, which proves the upper bound \eqref{eq:fub}.
  For the lower bound, we demonstrate a procedure to choose
  these edges that is guaranteed to give us a valid forward
  $\alpha$-switching.  
  Suppose that
  $e_0',\ldots,e_{k-1}'$ satisfy
  \begin{enumerate}[(a)]
    \item $e_i'$ is not contained in any short cycle;
      \label{li:a}
    \item the distance from $e_i$ to $e_i'$ is at least $r$;
      \label{li:b}
    \item the distance from $e_i'$ to $e_{i'}'$ is at least $r/2$;
      \label{li:c}
    \item the distance from $w_i$ to $u_i$ is at least $r$.
      \label{li:d}
  \end{enumerate}
  Then the switching is valid by an argument identical to the one in \cite{MWW},
  which we will reproduce for convenience.
  By (\ref{li:b}), for all $i$, neither $u_i$ nor $w_i$ is adjacent
  to $v_i$ (or to $v_{i'}$ for any $i'$), as required in the definition
  of a switching.  Let $G'$
  be the graph obtained by applying the switching.
  We need to check now that the switching is valid; that is,
  the only short cycle it creates or destroys is $\alpha$.
  
  Since $\alpha$ shares no edges with other short cycles, its deletion
  does not destroy any other short cycles.  Condition (\ref{li:a}) ensures
  that no short cycles are destroyed by removing $e_0',\ldots,e_{k-1}'$.
  The switching does not create any short cycles either:
  Suppose otherwise, and let $\beta$
  be the new cycle in $G'$.  It consists of paths in $G\cap G'$,
  separated by new edges in $G'$.
  Any such path in $G\cap G'$ must have length at least $r/2$, because
  \begin{itemize}
    \item if it starts and ends in $\alpha$ and has length less than $r/2$,
      then combining this path with a path in $\alpha$
      gives a short cycle in $G$ that intersects $\alpha$;
    \item if it starts in $\alpha$ and finishes in  $W=\{u_0,w_0,\ldots,u_{k-1},
      w_{k-1}\}$ and has length less than $r/2$, then combining this path
      with a path in $\alpha$ gives a path violating condition (\ref{li:b});
    \item if it starts at some $e_i'$ and ends at $e_{i'}'$ then
      it must have length $r/2$ by (\ref{li:c}) if $i'\neq i$, and
      by (\ref{li:a}) if $i'=i$.
  \end{itemize}
  Thus $\beta$ contains exactly one path in $G\cap G'$.  The remainder of
  $\beta$ must be an edge $u_iv_i$ or $w_iv_i$, impossible by (\ref{li:b}),
  or a path $u_iv_iw_i$, impossible by (\ref{li:d}).
  
  Now, we find the number of switchings that satisfy conditions
  (\ref{li:a})--(\ref{li:d}) to get a lower bound on $F_\alpha$.
  We will do this by bounding from above the number of switchings
  out of the $[n]_kd^k$ counted in \eqref{eq:fub}
  that fail each condition (\ref{li:a})--(\ref{li:d}).
  \begin{itemize}
    \item
      There are a total of $\sum_{j=3}^rj\Cy{j}$ edges in short
      cycles in $G$.  Choosing one of the edges $e_0',\ldots,e_{k-1}'$
      from these and the rest arbitrarily, there are at most
      $[n-1]_{k-1}d^{k-1}k\sum_{j=3}^r2j\Cy{j}$ switchings 
      that fail condition (\ref{li:a}).
    \item
      The number of edges of distance less than $r$ from some edge is at most
      $2\sum_{j=0}^r(d-1)^j-1=O((d-1)^r)$.  
      At most
      $[n-1]_{k-1}d^{k-1}kO\big((d-1)^r\big)$ switchings then fail
      condition (\ref{li:b}).
    \item By a similar argument,
      at most $[n]_{k-1}d^{k-1}k^2O\big((d-1)^{r/2}\big)$ switchings fail
      condition (\ref{li:c}).
    \item By a similar argument, 
      at most $[n]_{k-1}d^{k-1}kO\big((d-1)^r\big)$ switchings fail
      condition (\ref{li:d}).
  \end{itemize}
  Adding these up and combining $O(\cdot)$ terms, we find that at most
  \begin{align*}
    [n-1]_{k-1}d^{k-1}k\left(\sum_{j=3}^r2j\Cy{j}+O\big((d-1)^r\big)\right)
  \end{align*}
  switchings out of the original $[n]_kd^k$ fail conditions
  by (\ref{li:a})--(\ref{li:d}), establishing \eqref{eq:flb}.
\end{proof}
  
For backwards switchings, we give a similar upper bound, but we 
only give our lower bound in expectation.
\begin{lemma}\label{lem:backwardsswitchings}
  Let $G$ be a random $d$-regular graph on $n$ vertices, and let
  $\alpha$ be a cycle of length $k\leq r$
  in the complete graph $K_n$.  Then
  \begin{align}
    B_\alpha &\leq \bigl(d(d-1)\bigr)^k\label{eq:bup}\\
    \intertext{and}
    \E B_\alpha &\geq \bigl(d(d-1)\bigr)^k\left(
      1 - \frac{\Cl{bs}k(d-1)^{r-1}}{n}\right).\label{eq:blp}
  \end{align}
\end{lemma}
\begin{proof}
  The question this time is given $\alpha$,
  how many choices of oriented paths yield a valid
  switching?  For any fixed $\alpha$,
  there are at most $(d(d-1))^k$ choices of oriented paths, proving
  \eqref{eq:bup}.
  For the lower bound, let $B=\sum_{\beta}B_\beta$, where $\beta$
  runs over all cycles of length $k$ in the complete graph.  
  We will first show
  that
  \begin{align}
    B\geq \frac{[n]_k\bigl(d(d-1)\bigr)^k}{2k}\left(
      1 - \frac{4k\sum_{j=3}^rj\Cy{j} + O\bigl(
        k(d-1)^{r}\bigr)}{nd}\right).\label{eq:bigb}
  \end{align}
  As in Lemma~\ref{lem:forwardswitchings}, we give
  conditions that ensure
  a valid switching.  Let  $\beta=v_0\cdots v_{k-1}$, and suppose
  that the paths $u_iv_iw_i$ in $G$
  for $0\leq i\leq k-1$
  satisfy
  \begin{enumerate}[(a)]
    \item the edges $v_iu_i$ and $v_iw_i$ are not contained in any short cycles;
      \label{li:2a}
    \item for all $1\leq j\leq r/2$,
      the distance between the paths $u_iv_iw_i$ and $u_{i+j}v_{i+j}w_{i+j}$
      is at least $r-j+1$.\label{li:2b}
  \end{enumerate}
  Any choice of edges satisfying these conditions gives a valid
  backwards switching:
  Condition (\ref{li:2b}) ensures that $v_i\not\sim v_{i+1}$ and
  $w_i\not\sim u_{i+1}$, as required in the definition of a switching.
  Let $G'$ be the graph obtained by applying the switching.
  We need to check that no short cycles besides $\beta$ 
  are created or destroyed 
  by the switching.  By (\ref{li:2a}), none are destroyed.
  Suppose a short cycle $\beta'$ other than $\beta$ is created in $G'$.
  It consists of paths in $G\cap G'$, portions of $\beta$, 
  and edges $w_iu_{i+1}$.
  Any such path in $G\cap G'$ must have length at least $r/2$
  because
  \begin{itemize}
    \item if it starts at $u_i$, $v_i$, or $w_i$ and ends at
      $u_{i+j}$, $v_{i+j}$, or $w_{i+j}$ for $1\leq j\leq r/2$, then
      (\ref{li:2b}) implies this;
    \item if it starts and ends at one of $u_i$, $v_i$, and $w_i$, then
      (\ref{li:2a}) implies this.
  \end{itemize}
  Hence $\beta'$ must contain exactly one such path.
  The remainder of $\beta'$ must either be an edge $w_iu_{i+1}$,
  or a portion of $\beta'$, both of which are impossible by (\ref{li:2b}).
  
  There are $[n]_kd^k/2k$ choices for $\beta$, and at most
  $(d(d-1))^k$ choices for $u_i, w_i$, $0\leq i<k$.
  As before, we count how many of these potential switchings
  satisfy conditions (\ref{li:2a}) and (\ref{li:2b}) to get a lower
  bound on $B$.
  By similar arguments as in the proof of Lemma~\ref{lem:forwardswitchings}, 
  we find
  that at most
  \begin{align*}
    2[n-1]_{k-1}\big(d(d-1)\big)^{k-1}(d-1)\sum_{j=3}^rjC_j
  \end{align*}
  of the switchings violate condition (\ref{li:2a}), and
  at most $[n]_{k-1}\big(d(d-1)\big)^{k-1}
  O\big((d-1)^{r+1}\big)$ violate condition (\ref{li:2b}), which
  proves \eqref{eq:bigb}.
  
  By Proposition~\ref{prop:McKayestimate}\ref{item:onecycle}
  (or by \cite[eq.~2.2]{MWW}),
  \begin{align*}
    \E \Cy{k} &\leq \frac{\Cr{generalsubgraph}(d-1)^k}{2k}.
  \end{align*}
  Applying this to \eqref{eq:bigb} gives
  \begin{align*}
    \E B \geq \frac{[n]_k\bigl(d(d-1)\bigr)^k}{2k}
    \left( 1-O\left(\frac{k(d-1)^{r-1}}{n}\right)\right)
  \end{align*}
  
    By the exchangeability of the vertex labels of $G$, the law of
  $B_\beta$ is the same for all $k$-cycles $\beta$.
  It follows that $\E B = ([n]_k/2k)\E B_\alpha$, 
  proving \eqref{eq:blp}.
\end{proof}

\subsection{Applying Stein's method}
\label{subsec:stein2}
Rather than prove a theorem about the vector of cycle counts, we will give
a result on a more general process.
  Let $\Ii$ be an index set of possible
cycles in $K_n$
that the random graph $G$ might contain, and for $\alpha\in\Ii$,
let $I_\alpha$ be an indicator on $G$ containing $\alpha$.
We will show that the entire process
$(I_\alpha,\,\alpha\in\Ii)$ is well approximated by a vector
of independent Poissons, with the accuracy of the approximation
depending on the size of the set $\Ii$.  We will also prove a slight
variant in Proposition~\ref{prop:genpoiapprox} which achieves a better
error bound, at the expense of considering a less general process.
Our result on cycle counts, 
Theorem~\ref{thm:bestpoiapprox}, will follow easily from this.

Though we have no need for these process approximations in our paper,
similar results for the permutation model of random
graph have proven useful in \cite{JP}.  In any event, the machinery of
Stein's method gives them to us with no extra effort.

\begin{thm}\label{thm:genpoiapprox}
  Let $G$ be a random $d$-regular graph on $n$ vertices.
  For some collection $\Ii$ of cycles in the complete graph $K_n$
  of maximum length $r$,
  we define $\I=(I_\alpha,\,\alpha\in\Ii)$, with $I_\alpha=\1\{\text{$G$
  contains $\alpha$}\}$.  Let $\Z=(Z_\alpha,\,\alpha\in\Ii)$
  be a vector of independent Poisson random variables, with $\E Z_\alpha
  =(d-1)^{\abs{\alpha}}/[n]_{\abs{\alpha}}$, where $\abs{\alpha}$ denotes
  the length of the cycle $\alpha$.
  
  For some absolute constant $\Cr{52}$, for all $n$
  and $d,r\geq 3$ satisfying $r\leq n^{1/10}$ and $d\leq n^{1/3}$,
  \begin{align*}
    \dtv(\I,\,\Z) &\leq\sum_{\alpha\in\Ii}\frac{\Cl{52}\abs{\alpha}
    (d-1)^{\abs{\alpha}+r-1}}{n^{\abs{\alpha}+1}}.
  \end{align*}
\end{thm}
Before we give the proof, we show the result of applying this theorem
when $\Ii$ is all cycles of length $r$ or less:
\begin{cor}
  Let $G$ be a random $d$-regular graph on $n$ vertices,
  and let $\Ii$ be the collection of all cycles of length $r$
  or less in the complete graph $K_n$.  Define $\I$ and $\Z$ 
  as in the previous theorem.
  For some absolute constant $\Cr{53}$, for all $n$
  and $d,r\geq 3$,
  \begin{align*}
    \dtv(\I,\,\Z) &\leq
      \frac{\Cl{53}(d-1)^{2r-1}}{n}.
  \end{align*}
\end{cor}
\begin{proof}[Proof of the corollary]
  If $r>n^{1/10}$ or $d>n^{1/3}$, then 
  $\Cr{53}(d-1)^{2r-1}/n>1$ for a sufficiently large choice of
  $\Cr{53}$, and the total variation bound is trivial.
  Thus we can assume that this is not the case and apply
  the previous theorem:
  \begin{align*}
    \dtv(\I,\,\Z) &\leq \sum_{\alpha\in\Ii}\frac{\Cr{52}\abs{\alpha}
    (d-1)^{\abs{\alpha}+r-1}}{n^{\abs{\alpha}+1}}\\
      &= \sum_{k=3}^r\frac{[n]_k}{2k}\Bigl(\frac{\Cr{52}k
    (d-1)^{k+r-1}}{n^{k+1}}\Bigr)\\
    &=O\Bigl(\frac{(d-1)^{2r-1}}{n}\Bigr).\qedhere
  \end{align*}
\end{proof}
The strength of Theorem~\ref{thm:genpoiapprox} is that one can consider
a smaller set $\Ii$ of possible cycles and get a tighter total variation
bound.  For instance, if $\Ii$ is the set of all cycles in $K_n$ of length
$r$ or less containing vertex $1$, then $\I$ and $\Z$ are within
$O\big(r(d-1)^{2r-1}/n^2\big)$ in total variation norm.
\begin{rmk}\label{rmk:functional}
Since the cycle counts $(\Cy{3},\ldots,\Cy{r})$ are a functional
of $\I$, this corollary implies that
\begin{align*}
  \dtv\big((\Cy{3},\ldots,\Cy{r}),\,(Z_3,\ldots,Z_r)\big)
    &\leq \frac{\Cr{53}(d-1)^{2r-1}}{n},
\end{align*}
where $(Z_3,\ldots,Z_r)$ is a vector of independent Poisson random
variables with $\E Z_k = (d-1)^k/2k$.  In fact, we will give a slightly
better result in Theorem~\ref{thm:bestpoiapprox}.
\end{rmk}
\begin{proof}[Proof of Theorem~\ref{thm:genpoiapprox}]
  We will construct an exchangeable pair by taking a step in a reversible
  Markov chain.  To make this chain, define a graph $\GGG$
  whose vertices consist of all $d$-regular graphs on $n$
  vertices.
  For every valid forward $\alpha$-switching with $\alpha\in\Ii$ 
  from a graph
  $G_0$ to $G_1$, make an undirected edge in $\GGG$
  between $G_0$ and $G_1$. Place a weight of 
  $1/[n]_{\abs{\alpha}}d^{\abs{\alpha}}$ on each of these edges.
  The essential fact that will make our
  arguments work is that valid forward $\alpha$-switchings from
  $G_0$ to $G_1$
  are in bijective correspondence with valid backwards $\alpha$-switchings from
  $G_1$ to $G_0$.  Thus, we could have equivalently defined $\GGG$ 
  by forming an edge for every valid backwards switching.

  Define the degree of a vertex in a graph with weighted edges to be the
  sum of the adjacent edge weights.
  Let $d_0$ be the maximum degree of $\GGG$ as defined so far.
  To make $\GGG$ regular, add a weighted loop to each vertex that brings
  its degree up to $d_0$.
  Now, consider a random walk on $\GGG$ that moves with probability
  proportional to the edge weights.  This random walk is a Markov chain
  reversible with respect to the uniform distribution on $d$-regular
  graphs on $n$ vertices.  Thus, if $G$ has this distribution, and we
  obtain $G'$ by advancing one step in the random walk, 
  the pair of graphs $(G,G')$ is exchangeable.
  
  Let $I'_\alpha$ be an indicator on $G'$ containing
  the cycle $\alpha$, and define $\I'=(I'_\alpha,\,\alpha\in\Ii)$.
  It follows from the exchangeability of $G$ and $G'$ that $\I$ and
  $\I'$ are exchangeable, and we can apply
  Proposition~\ref{prop:steinsmethod} on this pair.  Define the events 
  $\Delta_\alpha^+$ and $\Delta_\alpha^-$
  as in that proposition.  By our construction,
  \begin{align*}
    \P[\Delta_\alpha^+\mid G]=\frac{B_\alpha}{d_0[n]_{\abs{\alpha}}
    d^{\abs{\alpha}}},\qquad
    \P[\Delta_\alpha^-\mid G]=\frac{F_\alpha}{d_0[n]_{\abs{\alpha}}
    d^{\abs{\alpha}}}.
  \end{align*}
  Thus by Proposition~\ref{prop:steinsmethod} with 
  all constants set to $d_0$,
  \begin{align}
    \dtv(\I,\,\Z) &\leq \sum_{\alpha\in\Ii}
       \E\left\lvert\frac{(d-1)^{\abs{\alpha}}}{[n]_{\abs{\alpha}}}
         -\frac{B_\alpha}{[n]_{\abs{\alpha}}d^{\abs{\alpha}}}\right\rvert
         +\sum_{\alpha\in\Ii}  
        \E\left\lvert I_\alpha
         -\frac{F_\alpha}{[n]_{\abs{\alpha}}d^{\abs{\alpha}}}\right\rvert.
       \label{eq:dtvbound}
  \end{align}
  We will bound these two sums.  Fix some $\alpha\in\Ii$, and let 
  $\abs{\alpha}=k$.  Applying first the upper bound and then
  the lower bound from Lemma~\ref{lem:backwardsswitchings},
  \begin{align}
    \E\left\lvert\frac{(d-1)^{k}}{[n]_{k}}
         -\frac{B_\alpha}{[n]_{k}d^{k}}\right\rvert
       &=\E\left[\frac{(d-1)^{k}}{[n]_{k}}
         -\frac{B_\alpha}{[n]_{k}d^{k}}\right]\nonumber\\
       &\leq \frac{\Cr{bs}k(d-1)^{k+r-1}}{n[n]_{k}}.\label{eq:E1}
  \end{align}
          
  To bound the other sum,  partition the state space of random regular
  graphs into three events:
  \begin{align*}
    A_1&=\{\text{$G$ does not contain $\alpha$}\},\\
    A_2&=\{\text{$G$ contains $\alpha$, which does not share
      an edge with another short cycle in $G$}\},\\
    A_3&=\{\text{$G$ contains $\alpha$, which shares an edge
    with another short cycle in $G$}\}.
  \end{align*}
  On $A_1$, we have $I_\alpha=F_\alpha=0$.  On $A_2$, both bounds
  from Lemma~\ref{lem:forwardswitchings} apply,
  giving us
  \begin{align*}
    \left\lvert I_\alpha
         -\frac{F_\alpha}{[n]_{k}d^{k}}\right\rvert
      &\leq \frac{2k\sum_{j=3}^rjC_j+\Cr{4}k(d-1)^r}{nd}.
  \end{align*}
  On $A_3$, we have $I_\alpha=1$ and $F_\alpha=0$.
  In all,
  \begin{align*}
    \E\left\lvert I_\alpha
         -\frac{F_\alpha}{[n]_{k}d^{k}}\right\rvert
       &\leq \E\left[\1_{A_2}\frac{2k\sum_{j=3}^rjC_j+\Cr{4}k(d-1)^r}{nd}
       +\1_{A_3}\right]\\
       &=\frac{2k}{nd}\E\biggl[\1_{A_2}\sum_{j=3}^rjC_j\biggr] 
          + \frac{\Cr{4}k(d-1)^r}{nd}
       \P[A_2] + \P[A_3].
  \end{align*}
    Let $\Jj$ be the set of all cycles of length $r$ or less in $K_n$
  that share no edges with $\alpha$.  On the event $A_2$, the graph $G$
  contains no cycles outside of this set (except for $\alpha$), and
  $\sum_{j=3}^r jC_j=k+\sum_{\beta\in\Jj}\abs{\beta}I_\beta$.
  Thus
  \begin{align}
    \E\left\lvert I_\alpha
         -\frac{F_\alpha}{[n]_{k}d^{k}}\right\rvert&\leq
         \frac{2k^2}{nd}\E\1_{A_2}+\frac{2k}{nd}\sum_{\beta\in\Jj}
         \abs{\beta}\E\1_{A_2}I_\beta+\frac{\Cr{4}k(d-1)^r}{nd}
       \P[A_2] + \P[A_3]\nonumber\\
       &\leq \frac{2k^2}{nd}\E{I_{\alpha}}+\frac{2k}{nd}\sum_{\beta\in\Jj}
         \abs{\beta}\E I_{\alpha}I_\beta+\frac{\Cr{4}k(d-1)^r}{nd}
       \E I_\alpha + \P[A_3].\label{eq:sums}
  \end{align}
  By Proposition~\ref{prop:McKayestimate}\ref{item:onecycle},
  \begin{align}
    \frac{2k^2}{nd}\E I_\alpha&=O\Bigl(\frac{k^2(d-1)^{k-1}}{n^{k+1}}\Bigr)
      \label{eq:part1}\\
    \intertext{and}
    \frac{\Cr{4}k(d-1)^r}{nd}
       \E I_\alpha&=O\Bigl(\frac{k(d-1)^{k+r-1}}{n^{k+1}}\Bigr)\label{eq:part3}.
  \end{align}
  By Proposition~\ref{prop:McKayestimate}\ref{item:twocycles} with $f=0$,
  for any $\beta\in\Jj$ we have $\E I_\alpha I_\beta\leq \Cr{generalsubgraph}(d-1)^{j+k}
    /n^{j+k}$, where $j=\abs{\beta}$.  For each $3\leq j\leq r$, there are at most $[n]_j/2j$ cycles
  in $\Jj$ of length $j$.  Therefore
  \begin{align}
    \frac{2k}{nd}\sum_{\beta\in\Jj}
         \abs{\beta}\E I_{\alpha}I_\beta
      &\leq \frac{2k}{nd}\sum_{j=3}^r\frac{[n]_j}{2j} \Bigl(\frac{j
      \Cr{generalsubgraph}(d-1)^{j+k}}{n^{j+k}}\Bigr)\nonumber\\
      &\leq \frac{k}{nd}\sum_{j=3}^r \frac{
      \Cr{generalsubgraph}(d-1)^{j+k}}{n^{k}}
      = O\Bigl(\frac{k(d-1)^{k+r-1}}{n^{k+1}}\Bigr).\label{eq:part2}
  \end{align}

  The last term of \eqref{eq:sums} is the most difficult to bound.
  Let $\Kk$ be the set of short cycles in $K_n$  that
  share an edge with $\alpha$, not including $\alpha$ itself.  
  By a union bound,
  \begin{align}
    \P[A_3]&\leq \sum_{\beta\in\Kk}\E I_\alpha I_\beta.\label{eq:A3ub}
  \end{align}
  Now, we classify and count the cycles $\beta\in\Kk$ according to 
  the structure of $\alpha\cup\beta$.
  Suppose that $\beta$ has length $j$, and consider
  the intersection of $\alpha$ and $\beta$ (the graph consisting
  of all vertices and edges contained in both $\alpha$ and $\beta$).
  Suppose this intersection graph has $p$ components and $f$ edges.
  As computed on \cite[p.~5]{MWW}, the number of possible isomorphism
  types of $\alpha\cup\beta$ given $p$ and $f$ is at most 
  $(2r^3)^{p-1}/(p-1)!^2$.  
  For each possible isomorphism type of $\alpha\cup\beta$, 
  there are no more than $2kn^{j-p-f}$ possible choices of $\beta$
  such that $\alpha\cup\beta$ falls into this isomorphism class.
  This is because $\alpha\cup\beta$ has $j+k-p-f$
  vertices, $k$ of which are determined by $\alpha$. In defining $\beta$,
  the remaining $j-p-f$ vertices can be chosen to be anything, and the
  intersection of $\alpha$ and $\beta$ can be rotated around $\alpha$ in
  $2k$ ways, all without changing the isomorphism class of $\alpha\cup\beta$.
  In all, we have shown that the number of $j$-cycles whose overlap with
  $\alpha$ has $p$ components and $f$ edges is at most
  \begin{align*}
    \frac{(2r^3)^{p-1}}{(p-1)!^2}
      2kn^{j-p-f}.
  \end{align*}
  
  For any such choice of $\beta$, we have
  $\E I_\alpha I_\beta
  \leq \Cr{generalsubgraph}(d-1)^{j+k-f}/n^{j+k-f}$ 
  by 
  Proposition~\ref{prop:McKayestimate}\ref{item:twocycles}.
  Applying this to \eqref{eq:A3ub},
  \begin{align}
    \P[A_3]&\leq\sum_{j=3}^r\sum_{p,f\geq 1} \frac{(2r^3)^{p-1}}{(p-1)!^2}
      2kn^{j-p-f}\frac{\Cr{generalsubgraph}(d-1)^{j+k-f}}{n^{j+k-f}}\nonumber\\
    &=\sum_{j=3}^r\sum_{p,f\geq 1} \frac{(2r^3)^{p-1}}{(p-1)!^2}
      \frac{2k\Cr{generalsubgraph}(d-1)^{j+k-f}}{n^{k+p}}\nonumber\\
    &=\sum_{j=3}^rO\Bigl(\frac{k(d-1)^{j+k-1}}{n^{k+1}}\Bigr)
    =O\Bigl(\frac{k(d-1)^{k+r-1}}{n^{k+1}}\Bigr).\label{eq:part4}
  \end{align}
  Combining \eqref{eq:part1}, \eqref{eq:part3}, \eqref{eq:part2}, and
  \eqref{eq:part4}, we have
  \begin{align*}
    \E\left\lvert I_\alpha
         -\frac{F_\alpha}{[n]_{k}d^{k}}\right\rvert&=
       O\Bigl(\frac{k(d-1)^{k+r-1}}{n^{k+1}}\Bigr).
  \end{align*}
  Applying this and \eqref{eq:E1} to \eqref{eq:dtvbound} establishes
  the theorem.
\end{proof}
As mentioned in Remark~\ref{rmk:functional}, we can apply this theorem
to give a total variation bound on the law of any functional
of $\I$.  This bound is often less than optimal, since this theorem
fails to exploit the $\lambda_k^{-1/2}$ factors in 
Proposition~\ref{prop:steinsmethod}.  We will take advantage
of these factors in the following proposition, and then apply this to
prove Theorem~\ref{thm:bestpoiapprox}.
\begin{prop}\label{prop:genpoiapprox}
  With the set-up of Theorem~\ref{thm:genpoiapprox}, divide up the
  collection of cycles $\Ii$ into bins $\Bb_1,\ldots,\Bb_s$. Let
  \begin{align*}
    I_k=\sum_{\alpha\in\Bb_k}I_\alpha, \qquad Z_k=\sum_{\alpha\in\Bb_k}Z_\alpha,
  \end{align*}
  and let $\lambda_k=\E Z_k$.
  Then
  \begin{align*}
    \dtv\big((I_1,\ldots,I_s),\,(Z_1,\ldots,Z_s)\big))
      &\leq \Cr{52}\sum_{k=1}^s\xi_k\sum_{\alpha\in\Bb_k}
      \frac{\abs{\alpha}
    (d-1)^{\abs{\alpha}+r-1}}{n^{\abs{\alpha}+1}},
  \end{align*}
  where $\xi_k=\min\big(1, 1.4\lambda_k^{-1/2}\big)$.
\end{prop}
\begin{proof}
  Define the exchangeable pair $(G,G')$ as in
  Theorem~\ref{thm:genpoiapprox}, and define
  $I'_1,\ldots,I'_s$ as the analogous quantities in
  $G'$.
  Define $\Delta_k^+$ and $\Delta_k^-$
  as in Proposition~\ref{prop:steinsmethod}, noting that
  \begin{align*}
    \P[\Delta_k^+\mid G] = \sum_{\alpha\in\Bb_k}\frac{B_\alpha}{d_0[n]_{\abs{\alpha}}
    d^{\abs{\alpha}}},\qquad\qquad
    \P[\Delta_k^-\mid G] = \sum_{\alpha\in\Bb_k}\frac{F_\alpha}{d_0[n]_{\abs{\alpha}}
    d^{\abs{\alpha}}}.
  \end{align*}
  By Proposition~\ref{prop:steinsmethod},
  \begin{align*}
    \dtv\big((I_1,\ldots&, I_s),\,(Z_1,\ldots,Z_s)\big)\\
    &\leq \sum_{k=1}^s
       \xi_k\left(\E\left\lvert\lambda_k-d_0\P[\Delta_k^+\mid G]\right\rvert+
         \E\left\lvert I_k-d_0\P[\Delta_k^-\mid G]\right\rvert\right)\\
    &=\sum_{k=1}^s
       \xi_k\E\Biggl\lvert\sum_{\alpha\in\Bb_k}\left(
       \frac{(d-1)^{\abs{\alpha}}}{[n]_{\abs{\alpha}}}-\frac{B_\alpha}
       {[n]_{\abs{\alpha}}d^{\abs{\alpha}}}\right)
       \Biggr\rvert\\
       &\qquad\quad\phantom{}+\sum_{k=1}^s\xi_k
         \E\left\lvert\sum_{\alpha\in\Bb_k}\left(I_\alpha-\frac{F_\alpha}
       {[n]_{\abs{\alpha}}d^{\abs{\alpha}}}\right)
       \right\rvert.
  \end{align*}
  These summands were already bounded in expectation in 
  Theorem~\ref{thm:genpoiapprox}, and applying these bounds proves
  the proposition.
\end{proof}
\begin{thm}\label{thm:bestpoiapprox}
  Let $G$ be a random $d$-regular graph on $n$ vertices
  with cycle counts $(\Cy{k},\,k\geq 3)$.
  Let $(Z_k,\,k\geq 3)$ be independent Poisson random variables
  with $\E Z_k=(d-1)^k/2k$.  For any $n\geq 1$ and $r,d\geq 3$,
  \begin{align*}
    \dtv\big((\Cy{3},\ldots,\Cy{r}),\,(Z_3,\ldots,Z_r)\big)
      &\leq\frac{\Cl{55}\sqrt{r}(d-1)^{3r/2-1}}{n}
   \end{align*}
\end{thm}
\begin{proof}
  If $d> n^{1/3}$ or $r>n^{1/10}$, then
  $\Cr{55}\sqrt{r}(d-1)^{3r/2-1}/n>1$ for a 
  sufficiently large choice of $\Cr{55}$,
  and the theorem holds trivially.  Thus we can assume that
  $d\leq n^{1/3}$ and $r\leq n^{1/10}$.
  
  Let $\lambda_k=(d-1)^k/2k$.
  With $\Ii_k$ defined as the set of all cycles in $K_n$ of length $k$,
  we apply the previous proposition with bins $\Ii_3,\ldots,\Ii_r$ to get
  \begin{align*}
    \dtv\big((\Cy{3},\ldots,\Cy{r}),\,(Z_3,\ldots,Z_r)\big)
      &\leq \Cr{52}\sum_{k=3}^r1.4\lambda_k^{-1/2}\sum_{\alpha\in\Ii_k}
      \frac{k
    (d-1)^{k+r-1}}{n^{k+1}}\\
    &=\sum_{k=3}^rO\Bigl(
      \frac{\sqrt{k}
    (d-1)^{k/2+r-1}}{n}\Bigr)\\
    &=O\Bigl(\frac{\sqrt{r}(d-1)^{3r/2-1}}{n}\Bigr).\qedhere
  \end{align*}
  
\end{proof}

\section{Eigenvalue fluctuations of random regular graphs}
\label{sec:eigenvalues}
Consider a random symmetric $n\times n$ matrix $X_n$
with eigenvalues $\lambda_1\geq\cdots\geq\lambda_n$.
As we mentioned in the introduction,
a \emph{linear eigenvalue statistic} is a random variable
of the form $\sum_{i=1}^nf(\lambda_i)$ for some function $f$.
A common problem in random matrix theory is to understand
the asymptotic behavior of linear eigenvalue statistics.
Typically, one shows convergence to a deterministic limit under one scaling (the
first-order behavior), and to a distributional limit under another scaling
(the second-order behavior).  The prototypical example is when
$X_n$ is a Wigner matrix: the first-order behavior is given by
Wigner's semicircle law (see \cite{BS} for a modern account
of Wigner's result), and for sufficiently smooth $f$,
the fluctuations from this are
normal \cite{SiS,BY}.

Recently, the problem of finding the fluctuations of
linear eigenvalue statistics was considered for 
random permutation matrices  \cite{BAD}, 
where for sufficiently smooth $f$,
the limiting distribution is non-Gaussian. 
This is striking because this behavior is  \emph{non-universal}.
The na\"ive expectation would have been that the eigenvalues of these matrices
should behave as in the \emph{Gaussian orthogonal ensemble}, which consists of Wigner matrices
with Gaussian entries.
In \cite{DJPP}, the same problem was considered
for the adjacency matrices of random
$d$-regular graphs drawn from the permutation model.
As with random permutation matrices,
for sufficiently smooth $f$,
the limiting fluctuations 
are non-Gaussian if
$d$ is a fixed constant. On the other hand, if $d$ grows to infinity with $n$,
the limiting fluctuations are Gaussian. The first-order behavior of linear
eigenvalue statistics shows the same dichotomy, with the non-universal Kesten-McKay limit
when $d$ is fixed replaced by the semicircle law when $d$ grows
with $n$ \cite{DP,TVW}.

Our goal is to extend these fluctuation results to the uniform
model of random regular graph.
Following the approach of \cite{DJPP}, 
we will use Theorem~\ref{thm:bestpoiapprox} to estimate the distribution
of counts of \emph{cyclically non-backtracking walks}.
Using a connection between these counts and the graph's eigenvalues,
we compute the non-Gaussian limiting fluctuations in Theorem~\ref{thm:dfixedlimit}.
We will then show
in Theorem~\ref{thm:dgrowslimit} that when $d$ grows with $n$,
the eigenvalue fluctuations converge to nearly the same limit
as in the GOE.

\renewcommand{\Cy}[2][\infty]{C_{#2}^{(#1)}}
\newcommand{\CNBW}[2][\infty]{\mathrm{CNBW}_{#2}^{(#1)}}

If a walk on a graph begins and ends at the same vertex, we
call it \emph{closed}.
We call a walk on a graph \emph{non-backtracking} if it never follows
an edge and immediately follows that same edge backwards.
Non-backtracking walks are also known as irreducible.

     \begin{figure}
      \begin{center}
        \begin{tikzpicture}[scale=1.8,vert/.style={circle,fill,inner sep=0,
              minimum size=0.15cm,draw},>=stealth]
            \draw[thick] (0,0) node[vert,label=right:1](s0) {}
                  -- ++(0,1cm) node[vert,label=below right:2](s1) {}
                  -- ++(0.707cm,0.707cm) node[vert,label=right:3](s2) {}
                  -- ++(-0.707cm,0.707cm) node[vert,label=right:4](s3) {}
                  -- ++(-0.707cm,-0.707cm) node[vert,label=left:5](s4) {}
                  --(s1);
        \end{tikzpicture}
      \end{center}
      \caption{The walk $1\to 2\to 3\to 4\to 5\to 2\to 1$ is non-backtracking,
      but not cyclically non-backtracking.}
      \label{fig:cnbw}
    \end{figure}
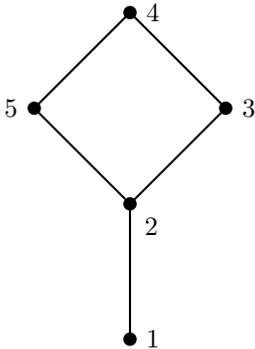  
 Consider a closed non-backtracking walk, and suppose that its last
 step is anything other than the reverse of its first step (that is, the
 walk does not look like the one given  in
 Figure~\ref{fig:cnbw}).  Then we call it a \emph{cyclically non-backtracking
 walk}.  These walks occasionally go by the name
 strongly irreducible.

Let $G_n$ be a random $d$-regular graph on $n$ vertices from the uniform
model,
and let $\Cy[n]{k}$ be the number of cycles of length $k$ in $G_n$.
  We define the random variable $\CNBW[n]{k}$ to be the number
 of cyclically non-backtracking walks of length $k$ in $G_n$.
Define $(\Cy{k},\, k\geq 3)$ to be
independent
Poisson random variables, with $\Cy{k}$ having mean $\lambda_k=(d-1)^k/2k$.
It will be convenient to define $\Cy{1}$, $\Cy{2}$, $\Cy[n]{1}$,
and $\Cy[n]{2}$
 as zero.

  Define
  \begin{align*}
    \CNBW{k} = \sum_{j\mid k} 2j\Cy{j}.
  \end{align*}
  For any cycle in $G_n$ of length $j$, where $j$ divides $k$, we obtain $2j$ 
  cyclically non-backtracking walks
  of length $k$ by choosing a starting point and direction and then
  walking around the cycle repeatedly.
  In fact, if $d$ and $k$ are small compared to $n$, 
  then these are likely
  to be the only cyclically non-backtracking walks of length~$k$ in $G_n$, as 
  the following proposition will show.
  
  \begin{prop}\label{prop:badwalks}
    Suppose $d\leq n^{1/3}$ and $k\leq n^{1/10}$.
    Let
    \begin{align*}
      \badw[n]{k} := \CNBW[n]{k}- \sum_{j\mid k} 2j\Cy[n]{j},
    \end{align*}
    the number of cyclically non-backtracking walks in the
    random $d$-regular graph $G_n$ that are not repeated walks around cycles.
    Then
    \begin{align*}
      \E \badw[n]{k} &\leq \frac{\Cl{badwalks}k^6(d-1)^k}{n}.
    \end{align*}
  \end{prop}
  \begin{proof}
    Call a cyclically non-backtracking walk \emph{bad} if it is not a repeated walk around
    a cycle.
    We just need to enumerate the possible bad walks and apply 
    Proposition~\ref{prop:McKayestimate}\ref{item:generalsubgraph} to bound
    the probability of each one.
    First, we give some notation first used in \cite{BroSha}.
    Let $v_0,\ldots, v_{k}\in\{1,\ldots,n\}$ satisfying $v_0=v_k$ be a sequence of vertices
    that forms a bad cyclically
    non-backtracking walk.
    Let $1\leq i\leq k$. We say that the $i$th step of the walk is
    \begin{itemize}
      \item \emph{free} if $v_i$ did not previously occur in the walk;
      \item a \emph{coincidence} if $v_i$ previously occurred in the walk,
        but the edge $v_{i-1}v_i$ did not;
      \item and \emph{forced} if the edge $v_{i-1}v_i$ previously occurred in the walk.
    \end{itemize}
    Let $\chi+1$ be the number of coincidences and $f$ the number of forced steps in the walk.
    With $v$ the number of vertices and $e$ the number of edges in the graph formed
    by the walk, we then have
    \begin{align*}
      v &= k-\chi-f,\\
      e &= k - f.
    \end{align*}
   It follows from the walk being bad that $\chi\geq 1$.
   
   \begin{clm}\label{clm:walkcounts}
     Consider walks on $K_n$ such that when the walk is viewed as a subgraph, all vertices
     have degree at most $d$. The number of such walks 
     with given values $\chi\geq 1$ and $f\geq 0$
     is at most
     $k^{3\chi+2}(d-1)^fn^{k-\chi-f}$.
   \end{clm}
   \begin{proof}[Proof of the claim]
     Imagine laying out the coincidences, then the forced steps,
     and then the free steps.
     Given that there are $\chi+1$ coincidences, there are
     $\binom{k}{\chi+1}\leq k^{\chi+1}$ possible subsets of indices
     $\{1,\ldots,k\}$ where the coincidences can occur.
     The vertex at a coincidence has already occurred in the walk,
     so there are fewer than $k$ choices for each of them,
     giving us a total of $k^{2\chi+2}$ choices so far.
     
     Forced steps can occur only after a coincidence or another forced step.
     After each coincidence, imagine assigning some number of the steps to be forced.
     The number of ways to do this is at most the number of weak compositions of $f$~elements
     into $\chi+1$ parts, $\binom{f+\chi}{\chi}$, which we can bound by $k^\chi$.
     At each forced step, the walk can only move along an edge that has already been traversed,
     so there are at most $d-1$ possible choices of vertices at each forced step.
     In all, this gives us at most $k^\chi(d-1)^f$ choices for the forced steps.
     
     At each of the $k-\chi-1-f$ free steps, we have at most $n$ choices of where to move, 
     and we have an additional $n$ choices for $v_0$, giving
     us another $n^{k-\chi-f}$ choices in all.
     Multiplying together these three bounds proves the claim.
   \end{proof}
   The probability of a given bad walk being found in $G_n$ is at most
   $\Cr{generalsubgraph}\bigl((d-1)/n\bigr)^{k-f}$
   by Proposition~\ref{prop:McKayestimate}\ref{item:generalsubgraph}.
   Applying Claim~\ref{clm:walkcounts} and summing over all possible bad walks,
   \begin{align*}
     \E \badw[n]{k} &\leq \sum_{\chi\geq 1}\sum_{f=0}^{k-1}
       k^{3\chi+2}(d-1)^fn^{k-\chi-f} \frac{\Cr{generalsubgraph}(d-1)^{k-f}}{n^{k-f}}\\
       &= \sum_{\chi\geq 1}\sum_{f=0}^{k-1} O\biggl(
         \frac{k^{3\chi+2}(d-1)^{k}}{n^\chi}
       \biggr)\\
       &= k^3(d-1)^k\sum_{\chi\geq 1}O\biggl(
         \frac{k^{3\chi}}{n^\chi}\biggr) = O\biggl(\frac{k^6(d-1)^k}{n}\biggr),
   \end{align*}
   completing the proof of Proposition~\ref{prop:badwalks}.
  \end{proof}

  \begin{cor}\label{thm:CNBWbound}
    \begin{align*}
      \dtv\left( \big( \CNBW[n]{k},\, 3\leq k\leq r  \big), 
        \big( \CNBW[\infty]{k},\, 3\leq k\leq r   \big)\right)
        \leq \frac{\Cl{11}\sqrt{r}(d-1)^{3r/2-1}}{n}.
    \end{align*}
  \end{cor}
\begin{proof}
  For any measurable function $f$ and random variables $X$ and $Y$,
  it holds that $\dtv(f(X),\,f(Y))\leq \dtv(X,Y)$.  
  It follows by Theorem~\ref{thm:bestpoiapprox} that
  \begin{align}
    \dtv \left( \bigg( \sum_{j|k}2j\Cy[n]{j},\, 3\leq k\leq r  \bigg), 
        \big( \CNBW[\infty]{k},\, 3\leq k\leq r   \big)\right)
        &\leq \frac{\Cr{55}\sqrt{r}(d-1)^{3r/2-1}}{n}.\label{eq:almosttv}
  \end{align}
  By the previous proposition,
  $\P[\badw[n]{k}\geq 1]\leq \Cr{badwalks}k^6(d-1)^k/n$.
  Summing these probabilities from $k=3,\ldots,r$,
  \begin{align}
    \bigg( \sum_{j|k}2j\Cy[n]{j},\,1\leq k\leq r  \bigg)=
    \big( \CNBW[n]{k},\,1\leq k\leq r  \big)\label{eq:vecequality}
  \end{align}
  with probability $1-O\bigl(r^6(d-1)^r/n\bigr)$.
  If two random variables
  are equal with probability $1-\eps$, then the total variation distance between their laws is at most $\eps$.  Thus
  the two random vectors in \eqref{eq:vecequality} have total variation
  distance $O\bigl(r^6(d-1)^r/n\bigr)$.  This fact and
  \eqref{eq:almosttv} prove the corollary.
\end{proof}

  \newcommand{\N}[2][\infty]{N_{#2}^{(#1)}}
To relate
Corollary~\ref{thm:CNBWbound}
to the eigenvalues of the adjacency matrix of $G_n$,
we define a set of polynomials
\begin{align*}
  \Gamma_0(x) & = 1, \\
  \Gamma_{2k}(x) &= 2 T_{2k}\left (\frac{x}{2} \right ) + \frac{d-2}{(d-1)^k}
  &\text{for $k \geq 1$,}\\
  \Gamma_{2k+1}(x) &= 2 T_{2k+1}\left(\frac{x}{2}\right)&\text{for $k \geq 0$.}
\end{align*}
Here
$\{T_n(x)\}_{n \in \mathbb{N}}$ are the Chebyshev polynomials of the first kind 
on the interval $[-1,1]$, defined inductively by
\begin{align*}
  T_0(x) &= 1,\\
  T_1(x) &= x,\\
  T_{n+1}(x) &= 2xT_n(x)-T_{n-1}(x),\quad n\geq 2.
\end{align*}

\begin{prop}[{\cite[Proposition~32]{DJPP}}]\label{prop:eigenlink}
  Let $A$ be the adjacency matrix of a (deterministic) $d$-regular graph $G$, and
  let $\lambda_1\ge\cdots\ge\lambda_n$ be the eigenvalues
  of $(d-1)^{-1/2}A$. Let $\mathrm{CNBW}_k$ be the number of cyclically
  non-backtracking walks of length~$k$ in $G$. Then
  \begin{align*}
    \sum_{i=1}^n\Gamma_k(\lambda_i)&=(d-1)^{-k/2}\mathrm{CNBW}_k.
  \end{align*}
\end{prop}
By Corollary~\ref{thm:CNBWbound}, we know the limiting distribution
of $\sum_{i=1}^nf(\lambda_i)$ when $f(x)=\Gamma_k(x)$.
The plan now is to extend this to a more general class of functions
by approximating by this
polynomial basis. We note the following bounds on the eigenvalues
of uniform random regular graphs.
\begin{prop}\label{prop:ebound}
  Let $G_n$ be a random $d$-regular graph on $n$ vertices
  with eigenvalues $\lambda_1\geq\cdots\geq\lambda_n$. Let
  $\lambda=\max_{i=2,\ldots,n}\abs{\lambda_i}$, the maximum nontrivial eigenvalue
  in absolute value.
  \begin{enumerate}[(a)]
    \item Suppose that $d\geq 3$ is fixed.  For any $\eps>0$, 
      \begin{align*}
        \P[\lambda>2\sqrt{d-1}+\eps]\to 0
      \end{align*}
      as $n\to\infty$.\label{item:dfixed}
    \item Suppose that $d=d(n)$ satisfies $d=o(n^{1/2})$.
      Then for some constant $K$, 
      \begin{align*}
        \P[\lambda>K\sqrt{d}]\leq\frac{\Cl{dgrows}}{n^2}
      \end{align*}
      for all $n$.
      \label{item:dgrows}
  \end{enumerate}
\end{prop}
\begin{proof}
  It is well known that (\ref{item:dfixed}) follows 
  from the results in \cite{Fri} by various contiguity
  results, but we cannot find an
  argument written down anywhere and will give one here.
  When $d$ is even, it follows from \cite[Theorem~1.1]{Fri} and
  the fact that for fixed $d$, permutation random graphs have no loops
  or multiple edges with probability bounded away from zero.
  This implies that the eigenvalue bound holds for permutation random
  graphs conditioned to be simple, and
  \cite[Corollary~1.1]{GJKW} transfers the result to the uniform model.
  When $d$ is odd (and $n$ even, as it has to be), 
  we apply \cite[Theorem~1.3]{Fri}, which gives
  the eigenvalue bound for graphs formed by superimposing $d$ random
  perfect matchings of the $n$ vertices. These are simple with probability
  bounded away from zero, and \cite[Corollary~4.17]{Wor99} transfers
  the result to the uniform model.
  
  Fact~(\ref{item:dgrows}) is proven in a more general context in
  \cite[Lemma~18]{BFSU}.
\end{proof}

 Following some facts from approximation theory, we will
 state the main result on the limiting distribution of linear
 eigenvalue statistics.
\begin{defn}
For $\rho>1$, let $E_\rho$ denote the image under the map $z\mapsto \frac{z+ z^{-1}}{2}$
of the open disc of radius~$\rho$ in the complex plane,
centered at the origin. We
call this the \emph{Bernstein ellipse} of radius~$\rho$. 
The ellipse has foci at $\pm 1$, and the sum
of the major semiaxis and the minor semiaxis is exactly
$\rho$. 
\end{defn}
\begin{prop}[{\cite[Theorem~8.1]{trefethen}}]\label{prop:approximationtheory}
  Suppose that $f\colon[-1,1]\to\RR$ can be analytically extended
  to $E_\rho$ and is bounded by $M$ there. Then $f$ has a unique expansion on $[-1,1]$ as
  \begin{align*}
    f(x) = \sum_{k=0}^{\infty}a_k T_k(x),
  \end{align*}
  and the coefficients of this expansion satisfy
  \begin{align*}
    \abs{a_0}\leq M,\qquad \abs{a_k}\leq\frac{2M}{\rho^k}.
  \end{align*}
\end{prop}
By applying the bound $\abs{T_k(x)}\leq 1$ and summing, we see that the approximations
$f_k(x) = \sum_{i=0}^k a_kT_k(x)$ satisfy
  \begin{align}
    \abs{f(x) - f_k(x)}\leq\frac{2M}{\rho^k(\rho-1)}\label{eq:chebapprox}
  \end{align}
  for $x\in[-1,1]$.

\begin{thm}\label{thm:dfixedlimit}
  Fix $d\geq 3$, and let $G_n$ be a random $d$-regular graph on
  $n$ vertices with adjacency matrix $A_n$.  
  Let $\lambda_1\geq\cdots\geq\lambda_n$ be the eigenvalues
  of $(d-1)^{-1/2}A_n$.
  
  Suppose that $f$ is a function such that $f(2z)$ is analytic on
  $E_\rho$, where $\rho=(d-1)^{\alpha}$
  for some $\alpha>3/2$.
  Then $f(x)$ can be expanded on $[-2,2]$ as
  \begin{align}
    f(x) = \sum_{k=0}^{\infty}a_k\Gamma_k(x),\label{eq:fexpansion}
  \end{align}
  and $Y_f^{(n)}:=\sum_{i=1}^nf(\lambda_i)-na_0$ converges
  in law as $n\to\infty$ to the infinitely divisible random variable
  \begin{align*}
    Y_f:=\sum_{k=1}^{\infty}\frac{a_k}{(d-1)^{k/2}}\CNBW{k}.
  \end{align*}
\end{thm}
\begin{proof}
  Let $f_k(x)=\sum_{i=0}^k a_i \Gamma_i(x)$.
  First, we show that $f_k(x)$ is a good approximation to $f(x)$.
  Applying Proposition~\ref{prop:approximationtheory} to $f(2x)$ gives
  the expansion \eqref{eq:fexpansion} and shows that
  \begin{align}
    \abs{a_k}\leq \Cl{M}(d-1)^{-\alpha k}\label{eq:coefficients}
  \end{align}
  for all $k\geq 1$.  
  On any interval $[-A,A]$ with $A>1$, the maximum of $\abs{T_k(x)}$ occurs
  at the endpoints. Using a well-known expression for $T_k(x)$, we have
  \begin{align}
    \max_{\abs{x}\leq A}\abs{T_k(x)} &= \frac{\bigl(A-\sqrt{A^2-1}\bigr)^k + \bigl(A + \sqrt{A^2-1}\bigr)^k}{2}.
    \label{eq:Tbound}
  \end{align}
  Applying this, one can see that for any $\delta>0$, it is possible to choose $\eps>0$
  such that for $\abs{x}\leq 2+\eps$,
  \begin{align*}
    \abs{\Gamma_k(x)} \leq (1+\delta)^k
  \end{align*}
  for all sufficiently large~$k$.
  Choosing $\delta$ small enough, this shows in combination with \eqref{eq:coefficients} that
  \begin{align}\label{eq:uniform}
    \sup_{\abs{x}\leq 2+\eps}\abs{f(x)-f_k(x)}\leq \Cl{M'}(d-1)^{-\alpha' k}
  \end{align}
  for some $\frac32<\alpha'<\alpha$.
  We also note that applying \eqref{eq:coefficients} and \eqref{eq:Tbound} in the same
  way with $A=d/2\sqrt{d-1}$ shows that $f_k\to f$ uniformly on $[-d/\sqrt{d-1},\,d/\sqrt{d-1}]$,
  which deterministically contains all the eigenvalues of $(d-1)^{-1/2}A_n$.

  The sum defining $Y_f$ converges almost
  surely, since it can be rewritten as
  \begin{align*}
    Y_f=\sum_{j=1}^{\infty}\sum_{i=1}^{\infty}\frac{a_{ij}}{(d-1)^{ij/2}}
    2j\Cy{j},
  \end{align*}
  and this is a sum of independent random variables, bounded in
  $L^2$ by \eqref{eq:coefficients}.
  Choose $\beta$ satisfying $\frac{1}{\alpha'}<\beta<\frac23$ and define
  \begin{align*}
    r_n&=\floor{\frac{\beta\log n}{\log(d-1)}},\\
    X_f^{(n)} &= \sum_{k=1}^{r_n}\frac{a_k}{(d-1)^{k/2}}\CNBW[n]{k}.
  \end{align*}
  We will use $X_f^{(n)}$ to approximate $Y_f^{(n)}$, noting that
  $X_f^{(n)}=\sum_{i=1}^nf_{r_n}(\lambda_i)-na_0$ by
  Proposition~\ref{prop:eigenlink}.
  By Corollary~\ref{thm:CNBWbound} and the fact that $\beta<\frac23$, the total variation distance
  between $X_f^{(n)}$ and $\sum_{k=1}^{r_n}(d-1)^{-k/2}a_k\CNBW{k}$
  vanishes as $n$ tends to infinity.  This sum converges
  almost surely to $Y_f$ as $n$ tends to infinity, 
  so $X_f^{(n)}$ converges in law
  to $Y_f$.  By Slutsky's Theorem, we need only show that
  $Y_f^{(n)}-X_f^{(n)}$ converges to zero in probability.
  
  Fix $\delta>0$.  We need to show that
  \begin{align*}
    \limn \P\left[\abs{Y_f^{(n)}-X_f^{(n)}} > \delta\right]=0.
  \end{align*}
  We have
  \begin{align*}
    \abs{Y_f^{(n)}-X_f^{(n)}}&\leq \sum_{i=1}^n \abs{f(\lambda_i)
      -f_{r_n}(\lambda_i)}.
  \end{align*}
  As noted before, $f_k(x)\to f(x)$ for any $\abs{x}\leq d/\sqrt{d-1}$.
  In particular, for the deterministic top eigenvalue
  $\lambda_1=d/\sqrt{d-1}$, we have $f_k(\lambda_1)\to f(\lambda_1)$.
  Thus
  $f(\lambda_i)
      -f_{r_n}(\lambda_i) < \delta/2$ for all sufficiently
  large $n$.
  
  Suppose that the remaining eigenvalues are contained in
  $[-2-\eps,2+\eps]$.  By \eqref{eq:uniform},
  \begin{align*}
    \sum_{i=2}^n \abs{f(\lambda_i)-f_{r_n}(\lambda_i)}
      &\leq  M(n-1)(d-1)^{-\alpha' r_n}
      \leq M n^{-\alpha'\beta+1},
  \end{align*}
  and this tends to zero since $\alpha'\beta>1$.
  For sufficiently large $n$, this sum is thus
  bounded by $\delta/2$.  We can conclude that
  for all large enough $n$,
  \begin{align*}
    \P\left[\abs{Y_f^{(n)}-X_f^{(n)}} > \delta\right]
      &\leq \P\left[\sup_{2\leq i\leq n}\abs{\lambda_i}\leq 2+\eps\right],
  \end{align*}
  and this tends to zero by Proposition~\ref{prop:ebound}\ref{item:dfixed}.
\end{proof}

In our next theorem, we extend this theorem to the case when the degree
grows with $n$. We will need a technical lemma on a normal approximation for the
Poisson distribution:
\begin{lemma}[Lemma~19 in \cite{Elliot}]\label{lem:Elliot}
  Suppose $X\sim\Poi(\lambda)$ and $W=(X-\lambda)/\sqrt{\lambda}$.
  Then $X$ can be coupled with $Z\sim N(0,1)$ so that
  $\E\lvert W-Z\rvert\leq 1/\sqrt{\lambda}$.
\end{lemma}
An entire function $f$ is said to be \emph{of order less than $m$} if
$\abs{f(z)}\leq \bigl\lvert e^{z^m}\bigr\rvert$ for all sufficiently large $z$.
Given a test function of order less than $m$, our theorem
gives conditions on the growth of the degree of the random regular graphs
such that the eigenvalue fluctuations converge to Gaussian.
The theorem also gives the limiting variance:
\begin{thm}\label{thm:dgrowslimit}
  Let $G_n$ be a random $d_n$-regular graph on $n$ vertices, 
  with $d_n\to\infty$ as $n\to\infty$.
  Let $A_n$ be the adjacency matrix of $G_n$, and 
  let $\lambda_1\geq\cdots\geq\lambda_n$ be the eigenvalues
  of $(d_n-1)^{-1/2}A_n$.
  
  Suppose that $f$ is entire with order less than $m$, which implies that
  it can be expressed as $f(x)=\sum_{k=0}^{\infty}a_k T_k(x/2)$.
  If $d_n\leq (\log n)^{\frac{2}{3m}-\eps}$ for some $\eps>0$, then
  \begin{align}
    \sum_{i=1}^n f(\lambda_i) - \E \sum_{i=1}^n f(\lambda_i)\label{eq:Yfn}
  \end{align}
  converges in law to normal with mean zero and variance
  $\frac12\sum_{k=3}^{\infty}ka_k^2$.
\end{thm}
\begin{proof}
  Note that the sum defining the limiting variance is finite
  by Proposition~\ref{prop:approximationtheory}.
  Choose $\beta$ satisfying $\frac23-m\eps<\beta<\frac23$, and let
  \begin{align*}
    r_n:=\floor{\frac{\beta\log n}{\log(d_n-1)}}.
  \end{align*}
  First, we show that the expression
  \begin{align}
    \sum_{k=3}^{r_n} \frac{ka_k}{(d_n-1)^{k/2}}\Bigl(\Cy{k}-\E\Cy{k}\Bigr).\label{eq:firstsum}
  \end{align}
  converges to the desired limit.
  Then, we will gradually change this expression while maintaining the same limit until
  we arrive at \eqref{eq:Yfn}.
  
  Let $\{Z_k\}_{k\in\NN}$ be i.i.d.\ standard Gaussians. By Lemma~\ref{lem:Elliot},
  this collection can be coupled with $\bigr\{\Cy{k}\bigl\}_{k\in\NN}$ so that
  \begin{align*}
    \E\biggl\lvert\frac{\sqrt{2k}}{(d_n-1)^{k/2}}\Bigl(\Cy{k}-\E\Cy{k}\Bigr)-Z_k\biggr\rvert
      \leq \frac{\sqrt{2k}}{(d_n-1)^{k/2}}.
  \end{align*}
  Thus the $L^1$ distance between \eqref{eq:firstsum}
  and $\sum_{k=3}^{r_n} \sqrt{k}a_kZ_k/\sqrt{2}$ is at most
  \begin{align*}
    \E\sum_{k=3}^{r_n} \frac{\sqrt{k}a_k}{\sqrt{2}}\Biggl\lvert
    \frac{\sqrt{2k}}{(d_n-1)^{k/2}}\Bigl(\Cy{k}-\E\Cy{k}\Bigr)-Z_k\Biggr\rvert
    &\leq \sum_{k=3}^{\infty}\frac{ka_k}{(d_n-1)^{k/2}},
  \end{align*}
  which vanishes as $n\to\infty$.
  This implies that \eqref{eq:firstsum} converges in law to a centered Gaussian
  with variance $\sum_{k=3}^{\infty}\frac{k}{2}a_k^2$.
  
  Now, we present some expressions and show that they converge to the same limit.
 \begin{align*}
   \sum_{k=3}^{r_n}\frac{a_k}{2(d_n-1)^{k/2}}\Bigl(\CNBW{k} - \E\CNBW{k}\Bigr)
   \tag*{\textbf{Expression 1:}}
 \end{align*}
  The difference between Expression~1 and \eqref{eq:firstsum} is
  \begin{align*}
    \sum_{k=3}^{r_n}\frac{a_k}{(d_n-1)^{k/2}}\sum_{\substack{j\mid k\\j<k}}j\Bigl(\Cy{j}-\E\Cy{j}\Bigr)
    &\leq \sum_{j=3}^{\infty}\sum_{i=2}^{\infty}\frac{ja_{ij}}{(d_n-1)^{ij/2}}\Bigl(\Cy{j}-\E\Cy{j}\Bigr)\\
    &= \sum_{j=3}^{\infty}O\biggl(\frac{ja_{2j}}{(d_n-1)^j}\biggr)\Bigl(\Cy{j}-\E\Cy{j}\Bigr),
  \end{align*}
  and the variance of this vanishes as $n\to\infty$.
  Thus the difference between Expression~1 and \eqref{eq:firstsum} converges to $0$ in probability.
  
 \begin{align*}
   \sum_{k=3}^{r_n}\frac{a_k}{2(d_n-1)^{k/2}}\Bigl(\CNBW[n]{k} - \E\CNBW{k}\Bigr)
   \tag*{\textbf{Expression 2:}}
 \end{align*}
 By Corollary~\ref{thm:CNBWbound} and our choice of $r_n$,
 the total variation distance between Expressions~1 and~2 vanishes as $n\to\infty$.

 \begin{align*}
   \sum_{k=3}^{r_n}\frac{a_k}{2(d_n-1)^{k/2}}\Bigl(\CNBW[n]{k} - \E\CNBW[n]{k}\Bigr)
   \tag*{\textbf{Expression 3:}}
 \end{align*}
 The difference between
  Expressions~2 and~3 is the deterministic quantity
  \begin{align}
    \sum_{k=3}^{r_n}\frac{a_k}{2(d_n-1)^{k/2}}\Bigl(\E\CNBW{k}-\E\CNBW[n]{k}\Bigr).\label{eq:23diff}
  \end{align}
  Using the decomposition
  \begin{align*}
    \CNBW[n]{k} &= \sum_{j\mid k}2j\Cy[n]{j} + \badw[n]{k}
  \end{align*}
  from Proposition~\ref{prop:badwalks}, we have
  \begin{align*}
    \E\CNBW{k}-\E\CNBW[n]{k}
      &= \sum_{j\mid k} \Bigl((d-1)^j-2j\E\Cy[n]{j}\Bigr)
        - \E\badw[n]{k}.
  \end{align*}
  By \cite[eq.~(2.2)]{MWW} and Proposition~\ref{prop:badwalks},
  this is $O\bigl(k^6(k+d)(d-1)^k/n\bigr)$.
  By our choice of $r_n$ and the fact that $a_k\to 0$ as $k\to\infty$,
  equation~\eqref{eq:23diff} vanishes as $n\to\infty$.
 \begin{align*}
   \sum_{i=1}^n f(\lambda_i) - \E \sum_{i=1}^n f(\lambda_i)
   \tag*{\textbf{Expression 4:}}
 \end{align*}
  Let $f_k := \sum_{i=0}^ka_kT_k(x/2)$.
  By Proposition~\ref{prop:eigenlink} and the fact that
  $\lambda_1$ is deterministic, Expression~3 is equal to
  \begin{align*}
    \sum_{i=2}^n f_{r_n}(\lambda_i) - \E \sum_{i=2}^n f_{r_n}(\lambda_i).
  \end{align*}
  Thus it suffices to show that $\sum_{i=2}^n \bigl(f(\lambda_i)-f_{r_n}(\lambda_i)\bigr)$
  vanishes in $L^1$.

  Let $E$ be the event that $\sup_{i=2,\ldots,n}\abs{\lambda_i}\leq K$, where $K$
  is the constant from Proposition~\ref{prop:ebound}\ref{item:dgrows} that makes
  $\P[E^C]\leq \Cr{dgrows}n^{-2}$.
  We have
  \begin{align*}
    \E\biggl\lvert\sum_{i=2}^n \bigl(f(\lambda_i)-f_{r_n}(\lambda_i)\bigr)\biggr\rvert
    &\leq \Ee_1+\Ee_2+\Ee_3,
  \end{align*}
  where
  \begin{align*}
    \Ee_1 &:=  \E\biggl[\1_E\sum_{i=2}^n\bigl\lvert f(\lambda_i)-f_{r_n}(\lambda_i)\bigr\rvert\biggr],\\
    \Ee_2 &:= \E\biggl[\1_{E^C}\sum_{i=2}^n\bigl\lvert f(\lambda_i)\bigr\rvert\biggr],\\
    \Ee_3 &:= \E\biggl[\1_{E^C}\sum_{i=2}^n\bigl\lvert f_{r_n}(\lambda_i)\bigr\rvert\biggr],
  \end{align*}
  and we need to show that these quantities vanish as $n\to\infty$.
  
  By Proposition~\ref{prop:approximationtheory},
  \begin{align}
    \abs{a_k} &\leq 2\Bigl[\sup_{z\in E_\rho}f(2z)\Bigr]\rho^{-k}
      \leq 2\exp\biggl[ \Bigl( \rho+\frac{1}{\rho}\Bigr)^m\biggr]\rho^{-k}
      \label{eq:akentirebound}
  \end{align}
  for sufficiently large $\rho$.
  By \eqref{eq:Tbound}, $\sup_{x\in[-K,K]}\abs{T_k(x/2)} = O(K^k)$.
  This gives us
  \begin{align*}
    \sup_{x\in[-K,K]} \abs{f(x) - f_k(x)}
      &\leq \sum_{i=k+1}^{\infty}\abs{a_iT_i(x/2)}
      = O\biggl(\exp\biggl[ \Bigl( \rho+\frac{1}{\rho}\Bigr)^m\biggr]
        \biggl(\frac{K}{\rho}\biggr)^{k+1}\biggr).
  \end{align*}
  Set $\rho=k^{1/m}$ to approximately optimize this, and substitute $k=r_n$ to get
  \begin{align*}
    \sup_{x\in[-K,K]}\abs{f(x)-f_{r_n}(x)} &\leq
      O(1) \exp\biggl(\frac{\beta\log n}{\log(d_n-1)} \Bigl( O(1) - \frac{\log\log n - \log\log(d_n-1)}{m}
        \Bigr)\biggr).
  \end{align*}
  By the condition $d_n\leq(\log n)^{\frac{2}{3m}-\eps}$,
  \begin{align*}
    \sup_{x\in[-K,K]}\abs{f(x)-f_{r_n}(x)} &\leq
      O(1) \exp\biggl(\frac{\beta\log n}{\log(d_n-1)} \Bigl( O(1) - \frac{\frac{1}{\frac{2}{3m}-\eps}\log(d_n-1) 
        - \log\log(d_n-1)}{m}
        \Bigr)\biggr)\\
        &=       O(1) \exp\biggl(
        O\Bigl(\frac{\beta\log n\log\log(d_n-1)}{m\log(d_n-1)}\Bigr) -
          \frac{\beta}{\frac23-m\eps}\log n
        \biggr)\\
        &= O(1) \exp\biggl(\log n\biggl(o(1)-
          \frac{\beta}{\frac23-m\eps}\biggr).
  \end{align*}
  Since $\beta>\frac23-m\eps$, this expression is $o(n^{-1})$, and
  \begin{align*}
    \Ee_1 &\leq (n-1)\sup_{x\in[-K,K]}\abs{f(x)-f_{r_n}(x)}\to 0
  \end{align*}
  as $n\to\infty$.

  Next, we show that $\Ee_2\to 0$. For large enough $d_n$, it holds that if $\abs{x}\leq d_n(d_n-1)^{-1/2}$,
  then
  \begin{align*}
    \abs{f(x)} &\leq \exp\biggl(\frac{d_n^m}{(d_n-1)^{m/2}}\biggr) = \exp\bigl(O\bigl(d_n^{m/2}\bigr)\bigr)
      = \exp\bigl(O\bigl((\log n)^{1/3}\bigr)\bigr).
  \end{align*}
  As $E^C$ occurs with probability at most $\Cr{dgrows}n^{-2}$,
  \begin{align*}
    \Ee_2 &\leq  O\bigl(n^{-2}\bigr) (n-1)\exp\Bigl(O\bigl((\log n)^{1/3}\bigr)\Bigr)\to 0
  \end{align*}
  as $n\to\infty$.
  
  Last, we consider $\Ee_3$.
  We apply \eqref{eq:akentirebound} with $\rho=k^{1/m}$ to show that for some $C$ depending on
  $m$ but not $k$,
  \begin{align*}
    \abs{a_k}\leq C^kk^{-k/m}
  \end{align*}
  for all $k\geq 1$.
  By \eqref{eq:Tbound}, for all $\abs{x}\leq d_n(d_n-1)^{-1/2}$,
  \begin{align*}
    \abs{T_k(x/2)} \leq d_n^{k/2}.
  \end{align*}
  Thus
  \begin{align*}
    f_{r_n}(x) &\leq a_0 + \sum_{k=1}^{r_n}\biggl(\frac{Cd_n^{1/2}}{k^{1/m}}\biggr)^k
      \leq a_0 + \sum_{k=1}^{\infty}\biggl(\frac{C(\log n)^{\frac{1}{3m}}}{k^{1/m}}\biggr)^{k}.
  \end{align*}
  Let $N_n=\floor{(2C)^m(\log n)^{1/3}}$, and break the sum into two pieces,
  one from $1$ to $N_n$ and the other from $N_n+1$ to $\infty$.
  Each term in the first piece is at most
  $C^{N_n}(\log n)^{N_n/3m}$, and a bit of analysis shows that
  \begin{align*}
    N_nC^{N_n}(\log n)^{N_n/3m} = o(n).
  \end{align*}
  The second piece is $o(1)$, as can be seen by comparing it to a geometric series.
  Thus
  \begin{align*}
    \Ee_3 \leq O\bigl(n^{-2}\bigr) (n-1)o(n)\to 0
  \end{align*}
  as $n\to\infty$.
\end{proof}

\begin{rmk}\label{rmk:variance}
  The only difference between the limiting distributions of
  Theorems~\ref{thm:dfixedlimit} and \ref{thm:dgrowslimit}
  and those of the permutation model of random graph
  in \cite{DJPP} derives from the slightly different
  expectations of $\CNBW{k}$ in the two models, and 
  from the fact that
  $\CNBW{1}=\CNBW{2}=0$ in the uniform model.
  The limiting variance in Theorem~\ref{thm:dgrowslimit} is the same
  as for eigenvalue fluctuations of the GOE, except that the coefficients
  $a_1$ and $a_2$ are ignored. (As the variance term for the GOE fluctuations
  can be expressed in many different ways, this is not entirely obvious. See
  Section~1.5 and in particular Proposition~3 from \cite{Elliot}.)
\end{rmk}

\subsection*{Acknowledgments}
The author gratefully acknowledges Ioana Dumitriu for pointing out
similarities between \cite{MWW} and \cite{DJPP}, 
and Soumik Pal and Elliot Paquette
for their general assistance.

\bibliographystyle{alpha}
\bibliography{uniform-stein}

\end{document}